\documentclass[a4paper,11pt]{article}
\usepackage{fullpage}
\usepackage[utf8]{inputenc}
\usepackage[english]{babel}
\usepackage[T1]{fontenc}
\usepackage{amsmath,amsthm}
\usepackage{amsfonts}
\usepackage{mathptmx}
\usepackage{amssymb}
\usepackage{float}
\usepackage{hyperref}
\usepackage{comment}
\usepackage{graphicx}
\usepackage{tikz}
\usepackage{mathtools}
\DeclareMathOperator*{\id}{id}
\DeclarePairedDelimiter\ceil{\lceil}{\rceil}

\newtheorem{teo}{Theorem} 
\newtheorem{cor}[teo]{Corollary}
\newtheorem{lem}[teo]{Lemma}
\newtheorem{pro}[teo]{Proposition}
\newtheorem{defe}[teo]{Definition}
\newtheorem{ex}[teo]{Example}

\newtheorem{rem}[teo]{Remark}
\newcommand{\resp}[1]{\ (resp. #1)}

\usepackage{xifthen}
\newcommand{\ifnv}[2]{\ifthenelse{\equal{#1}{}}{}{#2}}
\newcommand\sett[3][]{\left\{\left.#2\ifnv{#1}{\in #1}\vphantom{#3}\right|#3\right\}}
\newcommand{\card}[1]{\left|#1\right|}
\newcommand{\len}[1]{\left|#1\right|}
\newcommand{\abs}[1]{\left|\vphantom{f^f_f}#1\right|}

\newcommand{\db}{{\mathfrak d_{H}}}

\newcommand{\dl}{{\mathfrak d_{L}}}
\newcommand\dc{\mathfrak d_C}
\newcommand{\dll}{{d_L}}
\newcommand{\dhh}{d_{H}}
\newcommand\diam\theta
\newcommand\maxf[1]{\left\|{#1}\right\|}
\newcommand\minf[1]{\left|{#1}\right|}
\newcommand\cocy[2][]{s_{#2}^{#1}}
\newcommand{\uinf}[1]{#1^\infty}
\newcommand\maxal[1]{A^+_{#1}}
\newcommand{\N}{\mathbb N}
\newcommand{\Ns}{\mathbb N\setminus\{0\}}

\newcommand{\R}{\mathbb R}
\newcommand\M{\mathbb M}
\newcommand\G{\mathcal G}
\newcommand\dm{\mathfrak d_M}
\newcommand\dmm{d_M}
\newcommand\F{\mathcal F}

\usepackage{stmaryrd}
\newcommand{\co}[2]{\left\llbracket #1,#2\right\llbracket}
\newcommand{\oc}[2]{\left\rrbracket #1,#2\right\rrbracket}
\newcommand{\oo}[2]{\left\rrbracket #1,#2\right\llbracket}
\newcommand{\cc}[2]{\left\llbracket #1,#2\right\rrbracket}

\newcommand{\compl}[1]{{#1}^C}
\newcommand\maxd[1][f]{d_{#1}^+}
\newcommand\mind[1][f]{d_{#1}^-}

\usepackage{xcolor}

\newcommand{\ie}{\textit{i.e.}\ }

\newcommand\besi[1]{\mathfrak C_{#1}}
\newcommand\weyl[1]{\mathfrak S_{#1}}

\newcommand{\ipart}[1]{\left\lfloor #1\right\rfloor}

\usepackage{fancyhdr}
\begin{document}
\date{}
\title{Cellular automata and substitutions in topological spaces defined via edit distances}
\author{Firas BEN RAMDHANE\textsuperscript{1} and Pierre GUILLON\textsuperscript{2}}
\maketitle
\begin{center}
	\textsuperscript{1,2}Aix Marseille Univ, CNRS, Centrale Marseille, I2M, Marseille, France, \\
	\textsuperscript{1}Sfax University, Faculty of Sciences of Sfax, LMEEM, Sfax, Tunisia.\\
	\textsuperscript{1}\url{firas.benramdhane@fss.u-sfax.tn}. \\
	\textsuperscript{2}	\url{pguillon@math.cnrs.fr}.\\
\end{center}
\vspace{1cm}
\begin{abstract}
The Besicovitch pseudo-metric is a shift-invariant pseudo-metric on the set of infinite sequences, that enjoys interesting properties and is suitable for studying the dynamics of cellular automata.
They correspond to the asymptotic behavior of the Hamming distance on longer and longer prefixes.
Though dynamics of cellular automata were already studied in the literature, we propose the first study of the dynamics of substitutions.
We characterize those that yield a well-defined dynamical system as essentially the uniform ones.
We also explore a variant of this pseudo-metric, the Feldman pseudo-metric, where the Hamming distance is replaced by the Levenshtein distance. 
Like in the Besicovitch space, cellular automata are Lipschitz in this space, but here also all substitutions are Lipschitz.
In both spaces, we discuss equicontinuity of these systems, and give a number of examples, and generalize our results to the class of the dill maps, that embed both cellular automata and substitutions.
\end{abstract}

\section{Introduction}
In \cite{blanchard_cellular_1997} were studied the dynamics of cellular automata in the spaces of sequences endowed with the Besicovitch pseudo-metric, which is defined as the asymptotics of the Hamming distance over prefixes of the sequences.
This corresponds to the $\bar d$-metric defined for ergodic purposes in \cite{ornstein}.
\cite{feldman}, and independently \cite{katok}, proposed to replace the Hamming distance by the Levenshtein distance from \cite{levenshtein1966binary}, and get the $\bar f$-metric, which is useful in Kakutani equivalence theory.
The Levenshtein distance depends on the minimum number of the edit operations (deletion, insertion, substitution) required to change one word into an other word. 
It used extensively for information theory, linguistics, word algorithmics, statistics\ldots.
One can read some properties of the pseudo-metric in \cite[Chapter~2]{orw}, and a nice history of this notion in \cite{kwietniak}.

The recent \cite{garcia2020topological} can be seen as presenting a nice picture of those systems for which the identity map from the Cantor space into the Feldman space is a topological factor map, after a similar task has been achieved for the Besicovitsh space in \cite{garcia-ramos_weak_2014}.
Here, we adopt a complementary point of view, by considering the dynamics within the space itself.
Though the work on the Besicovitch space concerned mainly cellular automata so far, relaxing the pseudometric to edit space allows to naturally consider a larger class of systems, that also includes substitutions : the so-called \emph{dill maps}.


In Section~\ref{s:def}, we will introduce some basic vocabulary of symbolic dynamical systems, including 
dill maps, 
and define the Besicovitch and Feldman spaces. 
In Section~\ref{s:bes}, we study dill maps over the Besicovitch space, give a sufficient and necessary conditions for them to induce a well-defined dynamical system over this space, and give some examples of behaviours.
In Section~\ref{s:feld}, we do the same within the Feldman space. 

\section{Definitions and basic results}\label{s:def}
The aim of this section is to introduce some concepts and basic notations in symbolic dynamics, that will be used throughout this paper, and to introduce some symbolic dynamical objects and topological spaces. 

We start with some terminology in word combinatorics.
We fix once and for all an \emph{alphabet} $A$ of finitely many \emph{letters} (it will be precised in each example, but general in our statements).
A \emph{finite word} over $A$ is a finite sequence of letters in $A$; it is convenient to write a word as $u=u_{\co{0}{\len u}}$ to express $u$ as the concatenation of the letters $u_0,u_1,\ldots,u_{\len{u}-1}$, with $\len{u}$ representing the length of $u$, that is, the number of letters appearing in $u$ and $\co{0}{\len u}=\{0,...,|u|-1\}$.
The unique word of length $0$ is the empty word denoted by $\lambda$.
The number of occurrences of a letter $a\in A$ in a finite word $u$ is denoted by $\len{u}_a$.
An \emph{infinite word}  $x=x_0x_1x_2\ldots$ over $A$ is the concatenation of infinitely many letters from $A$.
The set of all finite \resp{infinite} words over $A$ is denoted by $A^*$ \resp{$A^\N$}, and $A^n$ is the set of words of length $n\in \N$ over $A$.

\subsection{Symbolic dynamics}
Let us now introduce some basic notions in symbolic dynamics. First we equip the set $A^\N$ with the product topology on each copy of $A$. The topology defined on $A^\N$ is equivalent to the topology defined by the \emph{Cantor distance} denoted by $\dc$ and defined as follows:
$$ \forall x\neq y\in A^\N, \dc(x,y)=2^{-\min\sett{n\in \N}{x_n\neq y_n}}, \forall x\neq y\in A^\N, \text{ and } \dc(x,x)=0, \forall x\in A^\N.$$
This space, called the \emph{Cantor} space, is compact, totally disconnected and perfect.

A (topological) \emph{dynamical system} is a pair $(X_d,F)$ where $F$ is a continuous map from a compact metric space $X_d=(X,d)$ to itself.
When $X_d$ is understood from the context, we may omit it.
In particular, the \emph{shift} dynamical system is the couple $(A^\N,\sigma)$, where $\sigma$ is the \emph{shift} defined for all $x\in A^\N$ by $\sigma(x)_i=x_{i+1}$, for $i\in \N$.

We can now introduce some topological properties of a dynamical system $(X_d,F)$. 
We say that $x\in X$ is a \emph{fixed point} if $F(x)=x$; it is \emph{periodic} if $F^n(x)=x$ for some $n>0$.
The map $F$ is $M$-\emph{Lipschitz}, for $M>0$, if $d(F(x),F(y))\leq Md(x,y)$ for all $x,y\in X$.

A point $x\in X$ is an \emph{equicontinuous point} of $(X_d,F)$ if: 
$$\forall \varepsilon>0,\exists \delta >0, \forall y\in X,d(x,y)<\delta\implies\forall t\in\N,  d(F^t(x),F^t(y))<\varepsilon.$$
The following property is a strong form of non-equicontinuity. A point $x\in X$ is \emph{$1^-$-unstable} if:
\[\forall \varepsilon>0, \forall \delta>0, \exists y\in X,d(x,y)<\delta\text{ but }\exists t\in\N, d(F^t(x),F^t(y))>1-\varepsilon.\]
A dynamical system $(X_d,F)$ is \emph{equicontinuous} if: 
$$\forall \varepsilon>0,\exists \delta >0, \forall x\in X, \forall y\in B(x,\delta),\forall t\in\N,  d(F^t(x),F^t(y))<\varepsilon.$$
Note that if $F$ is $M$-Lipschitz, then $F^t$ is $M^t$-Lipschitz.
It is then clear that if $F$ is $1$-Lipschitz, then $F$ is equicontinuous (and it is actually an equivalence, up to equivalent distance, as seen for instance in \cite[Proposition ~2.41]{kurka2003topological}).
A dynamical system $(X_d,F)$ is \emph{sensitive} if: 
$$\exists\varepsilon>0,\forall x\in X, \forall \delta>0, \exists y\in X,d(x,y)<\delta\text{ but }\exists t\in\N, d(F^t(x),F^t(y))>\varepsilon.$$
A dynamical system $(X_d,F)$ is \emph{expansive} if: 
$$\exists \varepsilon>0 ,\forall x\neq y\in X, \exists t\in\N, d(F^t(x),F^t(y))>\varepsilon.$$

As examples of dynamical systems, we will be interested in this paper by {cellular automata}, substitutions, and in general {dill maps}, which are defined from the Cantor space $X=A^\N$ to itself.
 For more details, we can refer to \cite{fogg2002substitutions}, \cite{cant} and \cite{kurka2003topological}. 

\subsection{Cellular automata}
\begin{defe}
A \emph{cellular automaton} (CA) with diameter $\diam$ is a map $F :A^\N\to A^\N$, such that there exists a map called \emph{local rule} $f:A^\diam \to A$ such that for all $x\in A^\N$, $i\in \N$: $F(x)_i=f(x_{\co{i}{i+\diam}}).$
\end{defe}
\begin{ex}
\begin{enumerate}
\item The shift which is a CA with diameter $\diam=2$ and local rule $f$ such that $f(ab)=b$.
\item Let $A=\{0,1\}$. The \emph{XOr} is the CA with diameter $\diam=2$ and local rule defined as follows:  
\[f(ab)=a+b \bmod  2, \text{ for every } a,b\in \{0,1\}.\]
\item Let $A=\{0,1\}$. The \emph{Min} CA with diameter $\diam=2$ and local rule $f$ defined as follows: 
\[f(ab)=\min\{a,b\}, \forall a,b\in A.\]
\end{enumerate}
\end{ex}
In the Cantor space, an elegant characterization of cellular automata was was given by Hedlund in \cite{hedlund1969endomorphisms} as follows: 
A function $F:A^\N\to A^\N$ is a cellular automaton if and only if it is a continuous function with respect to the metric Cantor space and shift-invariant (\ie $\forall x\in A^\N,F(\sigma(x))=\sigma(F(x))$).

\subsection{Substitutions} 
\begin{defe}
\begin{enumerate}
\item A \emph{substitution} $\tau$ is a nonerasing homomorphism of monoid $A^*$,
(\ie $\tau^{-1}(\lambda)=\{\lambda\}$ and $\tau(uv)=\tau(u)\tau(v)$, for all $u,v\in A^*$).
\item Substitution $\tau$ yields a dynamical system, denoted by $\overline\tau$, and defined over $A^\N$ by: $$\overline{\tau}(z)=\tau(z_0)\tau(z_1)\tau(z_2)\tau(z_3)\ldots,\forall z\in A^\N.$$
\item The \emph{lower norm} $\minf{\tau}$ and \emph{upper norm} $\maxf{\tau }$ of $\tau$ are defined by: 
$$\minf{\tau}=\min\sett{\len{\tau(a)}}{a\in A}\text{ and }\maxf{\tau}=\max\sett{\len{\tau(a)}}{a\in A}.$$
We say that  $\tau$ is \emph{uniform} if $\minf{\tau}=\maxf{\tau}.$
\item The matrix $M(\tau)=(M(\tau)_{ab})_{a,b\in A}$ is defined such that $M(\tau)_{ab}$ is the number $\len{\tau(a)}_b$ of occurrences of $b$ in $\tau(a)$.
$M_\tau$ can be written in block-triangular form.
An \emph{irreducible component} (or matrix block) is a maximal subalphabet $A'$ such that for all $a,b\in A'$, $b$ appears in $\tau^t(a)$, for some $t\in\N$.
A letter $a\in A$ \emph{can reach} a component $ A''$ if some $b\in A''$ appears in $\tau^t(a)$, for some $t\in\N$.
A {component} $ A'$ \emph{can reach} another one $ A''$ if some $a\in A'$ can reach it.
A component is \emph{terminal} if it cannot reach any distinct component.
A component is \emph{maximum} if the spectral radius of the corresponding submatrix is maximal, among all components.
We will denote $\maxal\tau\subset A$ the union of all components which can reach a maximal component. 
$\tau$ is \emph{irreducible} if there is only one irreducible component.
A classical object of study si the \emph{primitive} substitutions, which are irreducible, as well as all their powers (equivalently its matrix $M(\tau)$ is primitive, \ie $\exists n\in \N$, $\forall a,b\in A$ we have $M^n(\tau)_{ab}>0$).
\end{enumerate}
\end{defe}
\begin{ex}\label{x:subst}
Let $A=\{0,1\}$.
\begin{enumerate}
\item The \emph{Thue-Morse} substitution defined over $A$ by: 
\begin{center}
$\begin{array}{r c l}
\tau : 0 & \mapsto & 01 \\
1 & \mapsto & 10
\end{array}
\qquad
M(\tau)=
\begin{bmatrix}
1&1\\
1&1
\end{bmatrix}$
\end{center}
This is an primitive uniform substitution.
\item The \emph{Fibonacci} substitution defined over $A$ by: 
\[\begin{array}{r c l}
\tau : 0 & \mapsto & 01 \\
1 & \mapsto & 0
\end{array}
\qquad
M_\tau=
\begin{bmatrix}
0&1\\
1&1
\end{bmatrix}\qquad
M_\tau^2=
\begin{bmatrix}
1&1\\
1&2
\end{bmatrix}\]
This is an primitive nonuniform substitution: $\minf{\tau}=1<2=\maxf{\tau}$.
\item The \emph{doubling} substitution defined over $A$ by: 
\[\begin{array}{r c l}
\tau : 0 & \mapsto & 00 \\
1 & \mapsto & 11
\end{array}\]
This is a uniform reducible substitution: $\{0\}$ and $\{1\}$ are two disjoint invariant subalphabets.
\item\label{i:cantor}
	A uniform substitution $\tau$ is T\oe plitz if $\exists i\in\co0{\maxf\tau},\forall a,b\in A,\tau(a)_i=\tau(b)_i$.
An example is the Cantor substitution, defined over $A$ by: 
\begin{center}
	$\begin{array}{r c l}
		\tau : 0 & \mapsto & 010 \\
		1 & \mapsto & 111
	\end{array}$
\end{center}
This is a reducible uniform substitution: $\{1\}$ is an invariant subalphabet. 
\end{enumerate}
\end{ex}

The following is folklore, direct consequence from Perron-Frobenius theory (see for example \cite[Theorem~A.72]{kurka2003topological}).
\begin{rem}\label{r:pfrob}
	Let $\tau$ be a substitution and $t\in\N$.
	Clearly, $M(\tau)^t_{a,b}$ is the number of occurrences of $b$ in $\tau^t(a)$, and $\sum_{b\in\co0n}M(\tau)^t_{a,b}=\len{\tau^t(a)}$.
	Let $\tau$ be a substitution, $a\in A$, and $\rho_a$ the maximal spectral radius of the irreducible components that $a$ can reach.
	Then there exist $\alpha_a\in\R_+^*$ such that $\len{\tau^t(a)}=\sum_{b\in A}M(\tau)^t_{a,b}\sim_{t\to\infty}\alpha_a\rho_a^t$.
	In particular, $\maxal\tau$ is exactly the set of letters $a\in A$ such that $\rho_a$ is the spectral radius $\rho_+$ of the matrix, \ie the letters that grow the fastest.
\end{rem}
%

\subsection{Dill maps}
The dill maps were defined in \cite{salo2015block}, and generelize both substitutions and CA. 
Here we give a simple definition, which is equivalent to \cite[Definition~2]{salo2015block}.
\begin{defe}~\begin{itemize}
\item A \emph{dill map} $F$ with diameter $\diam\in \Ns$ is a dynamical system over the set of infinite words such that there exists a \emph{local rule} $f~:~A^{\diam}\to A^+$ satisfying:
\[\forall x\in A^\N,F(x)=f(x_{\co{0}{\diam}})f(x_{\co{1}{\diam+1}})f(x_{\co{2}{\diam+2}})\cdots.\]
\item The \emph{lower norm} $\minf{f}$ and the \emph{upper norm} $\maxf{f }$ of a dill map $F $ with diameter $\diam$ and local rule $f$ are defined by: 
$$\minf f=\min\sett{\len{f(u)}}{u\in A^\diam} \text{ and } \maxf f=\max\sett{\len{f(u)}}{u\in A^\diam}.$$
\item
We extend the local rule into a self-map $f^*:A^*\to A^*$ by: 
\[f^*(u)=
f(u_{\co{0}{\diam}})f(u_{\co{1}{1+\diam}})\ldots f(u_{\co{l}{l+\diam}}),\]
for $u$ such that $\len u\ge \diam$ and $f^*(u)=\lambda$ if $\len u<\diam$.
\item
We also consider the \emph{cocycle} $\cocy[n]{x}=\sum_{j\leq n}\len{f(x_{\co{j}{j+\diam}})}=\len{f^*(x_{\co{0}{n+\diam}})}$ for $n\in\N$ and $x\in A^\N$, which represents the position in which one can read the image at offset $n$.
\item
If $\maxf f=\minf f$, then we say that $F$ is \emph{uniform}.
In that case, the cocycle $\cocy[n]x$ does not depend on the infinite word $x$, and we note it $\cocy[n]{}$.
\end{itemize}\end{defe}
When it is clear from the context, we may identify a dill map with its local rule.
\begin{rem}~
\begin{enumerate}
\item The substitutions are the dill maps with diameter $\diam=1$.
\item The cellular automata are the uniform dill maps with $\minf f=\maxf f=1$. 
\item The composition of a substitution $\tau$ and a cellular automaton local rule $f$ with diameter $\diam$ is a dill map local rule $\tau\circ f$ with diameter $\diam$.
Actually, every dill map is the compositions of a substitution and a shift homomorphism (which is like a cellular automaton, but allowing to change the alphabet).
\end{enumerate}
\end{rem}
\begin{ex}
Let $f$ be the local rule of the XOr CA and $\tau$ be the Fibonacci substitution. Then $\tau\circ f$ is a local rule of a dill map with diameter $2$ and defined as follows:
\begin{center}
$
\begin{array}{r c l}
\tau\circ f: 00,11 &\mapsto & 01 \\
10,01 &\mapsto & 0
\end{array}$
\end{center}
\end{ex}
\begin{rem}
For all $x\in A^\N$, $n\mapsto\cocy[n]x$ is a one-to-one function.
\end{rem}
\begin{proof}
	For $n\leq m\in \N$ such that $\cocy[n]{x}=\cocy[m]{x}$ we have: 
$$\sum_{i\leq m}\len{f(x_{\co{i}{i+s}})}-\sum_{i\leq n}\len{f(x_{\co{i}{i+s}})}=\sum_{n< i\leq m}\len{f(x_{\co{i}{i+s}})}=0$$
Then for all $n< i\leq m$ we have $\len{f(x_{\co{i}{i+s}})}=0$.
Since $f$ is nonerasing then $n=m$. 
\end{proof}
Similarly to the case of cellular automata, we give a characterization of dill maps à la Hedlund.
\begin{teo}\label{dill cant}
A function $F:A^\N\to A^\N$ is a dill map if and only if it is continuous over the Cantor space and there exists a continuous map $s:A^\N\to \N;x\mapsto\cocy x$ such that for all $x\in A^\N$: $F(\sigma(x))=\sigma^{\cocy{x}}(F(x))$.
\end{teo}
\begin{proof}~\begin{itemize}
\item[$"\Rightarrow"$] Let $F$ be a dill map with diameter $\diam$ and local rule $f$. For $x\in A^\N$, $\varepsilon=2^{-p}$ for $p\in \N^*$ we take $m=\min\sett{i\in \N}{\cocy[i]{x}\geq p}$.
For $\delta=2^{-m}$ and $y\in A^\N$ such that $\dc(x,y)\leq \delta$ we have $x_{\co{0}{m}}=y_{\co{0}{m}}$.
Then $f^*(x_{\co{0}{m}})=f^*(y_{\co{0}{m}})$. Hence $F(x)_{\co{0}{p}}=F(y)_{\co{0}{p}}.$
So, $\dc(F(x),F(y))\leq 2^{-p}=\varepsilon$.
In conclusion, $F$ is continuous.
Now let us define $s(x)=\cocy{x}=\len{f(x_{\co{0}{\diam}})}$ for all $x\in A^\N.$
Let $x\in A^\N$ and $\varepsilon>0$. 
For $y\in A^\N$ such that $\dc(x,y)< 2^{-\diam}$ we have $x_{\co{0}{\diam}}=y_{\co{0}{\diam}}$.
Then $f(x_{\co{0}{\diam}})=f(y_{\co{0}{\diam}})$ and hence $\len{f(x_{\co{0}{\diam}})}=\len{f(y_{\co{0}{\diam}})}$.
So $d(\cocy{x},\cocy{y})=0$.
In conclusion, $s$ is continuous, and it satisfies $F(\sigma(x))=\sigma^{\cocy{x}}(F(x))$.\\
\item[$"\Leftarrow "$] We suppose now that there exists a continuous map $s:A^\N\to \N$.
We can write $A^\N=\bigcup_{n\in \N}s^{-1}(\{n\})$. 
Since $\{n\}$ is clopen set for all $n\in \N$ and $s$ is continuous, then $s^{-1}(\{n\})$ is open for all $n\in \N$.
On the other hand, $(A^\N, \dc)$ is compact, so that there exists a finite set $I\subset\N$ such that $A^\N=\bigcup_{i\in I}s^{-1}(\{i\})$.
Let $x\in A^\N$; then there exists $i\in I$ such that $x\in s^{-1}(n_i)$. 
Since $F$ is continuous, for $\varepsilon=2^{-n_i}$, there exists $r_i$ such that for all $y\in A^\N$ verifying $\dc(x,y)<2^{-r_i}$ we have $\dc(F(x),F(y))< \varepsilon$.
Then for all $x\in A^\N$, there exists $i\in \{1,\cdots,p\}$ and $r_i$ such that for all $y\in A^\N$ verifying $\dc(x,y)<2^{-r_i}$ we have $\dc(F(x),F(y))< 2^{-n_i}$.
Let $\diam=\max\sett{r_i}{i\in \{1,\cdots,p\}}$. 
Hence for all $x\in A^\N$, there exists $i\in \{1,\cdots,p\}$ such that for all $y\in A^\N$ with $x_{\co{0}{\diam}}=y_{\co{0}{\diam}}$, we have $F(x)_{\co{0}{n_i}}=F(y)_{\co{0}{n_i}}$ with $n_i=\cocy{x}$. \\
So there exists a map $f~:~A^\diam\to A^*$ such that $f(x_{\co{0}{\diam}})=F(x)_{\co{0}{n_i}}$, for all $x\in A^\N$.\\
On the other hand, for $x\in A^\N$ and $j\in \N$ we have:
$$f(x_{\co{j}{j+\diam}})=f(\sigma^j(x)_{\co{0}{\diam}})=F(\sigma^j(x))_{\co{0}{n_i}}, \text{ with } n_i=\cocy{\sigma^j(x)}.$$
Hence, by induction we find: 
$$F(\sigma^j(x))=F(\sigma(\sigma^{j-1}(x)))=\sigma^{\cocy{\sigma^{j-1}(x)}}F(\sigma^{j-1}(x))=\sigma^{m}F(x), \text{ with } m=\sum_{h=0}^{j-1}\cocy{\sigma^h(x)}.$$
Then : $f(x_{\co{j}{j+\diam}})=\sigma^{m}F(x)_{\co{0}{n_i}}=F(x)_{\co{m}{m+n_i}}.$\\
Finally, we find that $F(x)=f(x_{\co{0}{\diam}})f(x_{\co{1}{1+\diam}})f(x_{\co{2}{2+\diam}})\cdots.$
In conclusion, $F$ is a dill map.
\popQED\end{itemize}\end{proof}
\begin{cor}\label{shift inv}
Let $F$ be a dill map with local rule $f$. Then for all $x\in A^\N$, all $n\in \N$ we have:
$$F\circ\sigma^n(x)=\sigma^{\cocy[n]{x}} \circ F(x).$$
\end{cor}
\begin{proof}
Let $F$ be a dill map with diameter $\diam$ and local rule $f$. Then according to Theorem \ref{dill cant} there exists a continuous map $s: A^\N\to \N$ such that for all $x\in A^\N$: $F(\sigma(x))=\sigma^{\cocy{x}}(F(x))$.
We aim to prove by induction that for all $x\in A^\N$, for all $n\in \N$ there exists $\cocy[n]{x}\in \N$ such that: 
$$	F\circ\sigma^n=\sigma^{\cocy[n]{x}} \circ F.$$ 
Let $x\in A^\N$. For $k=1$ we have $F(\sigma(x))=\sigma^{\cocy{x}}(F(x))$. 
We suppose that our statement is true for $k=n$, so there exists $\cocy[n]{x}\in \N$ such that $F\sigma^n(x)=\sigma^{\cocy[n]{x}} F(x)$.
Now, for $k=n+1$ we have:
$$F\circ\sigma^{n+1}(x)=F(\sigma(\sigma^n(x)))=\sigma^{\cocy{\sigma^n(x)}}(F(\sigma^n(x))) =\sigma^{\cocy{\sigma^n(x)}}(\sigma^{\cocy[n]{x}}(F(x)))= \sigma^{\cocy[n+1]{x}}(F(x)).$$
Which is the next step of the induction hypothesis.\\
Hence, for all $x\in A^\N$, all $n\in \N$ and for $\cocy[n]{x}=\sum_{k=0}^{n-1}\cocy{\sigma^k(x)}$ we have: 
$$F(\sigma^n(x))=\sigma^{\cocy[n]{x}}(F(x)).\popQED$$
\end{proof}
\subsection{The Besicovitch space}
In this subsection, we recall the definition and topological properties of Besicovitch space. 

Recall that a distance is an application over $A^*\times A^*$ to $\R_+$ satisfying: separation, symmetry, and the triangle inequality.
We can endow the set $A^n$ for $n\in \N$ with distance.
The prototypical example is the \emph{Hamming distance} denoted by $d_H$.
It is usually defined as the number of differences between two finite words of the same length.
Let us present a definition in terms of edit distance, that is a number of operations from a specific kind to transfor a word into another one.
\begin{defe}\label{d:hamming}~\begin{itemize}
\item The \emph{substitution} operations $S^a_j$ at position $j\in\co0{\len u}$, for $a\in A$, are defined over finite word $u\in A^*$ as follows:
$S^a_j(u)=u_0\cdots u_{j-1}au_{j+1}\cdots u_{\len u-1}$.
\item	Between two finite words with the same length $u,v$, we define the \emph{Hamming distance}:
\[ d_H(u,v)=\min \sett{m\in \co{0}{\len u}}{\exists j_1<j_2<\cdots <j_m, (a_j)_{j\le m}\subset A^m,S^{a_1}_{j_1}\circ \cdots S^{a_n}_{j_m}(u)=v}.\]
\end{itemize}\end{defe}
Of course, the definition would be completely equivalent by allowing substitutions to be performed on both $u$ and $v$, and is simply the number of differences, letterwise.

\begin{rem}
	Note that this distance is an \emph{additive distance} \ie for all $u,u',v,v'$ such that $\len u=\len{u'}$ and $\len v=\len{v'}$,
 \[d_H(uv,u'v')=d_H(u,u')+d_H(v,v').\]
\end{rem}

\begin{defe}
The Besicovitch pseudo-metric, denoted by $\db$, is defined as follows:
$$\forall x,y\in A^\N, \db(x,y)=\limsup_{l\to \infty} \frac{d_H(x_{\co{0}{l}},y_{\co{0}{l}})}{l}.$$
\end{defe}
It is easy to verify that this is a pseudo-metric: it is symmetric, zero over diagonal pairs, and satisfies the triangular inequality.
On the other hand, it is not a distance since we can find two different infinite words where the pseudo-metric is worth zero (for instance, we can take two infinite words with finitely many of differences).
Hence, it is relevant to quotient the space of infinite words by the equivalence of zero distance, in order to get a separated topological space, as mentioned in the following definition:
\begin{defe}
\begin{itemize}
	\item The relation  $x \sim_{\db} y \iff \db(x,y)=0$, is an equivalence relation.
	\item The quotient space $A^\N_{/\sim_{\db}}$ is a topological space, called the Besicovitch space denoted by $X_{\db}$.
	\item We denote by $x_{\db}$ the equivalence class of $x\in A^\N$ in the quotient space.
	\item Any map $F:A^\N\mapsto A^\N$ such that $\db(x,y)=0\implies\db(F(x),F(y))$ induces a well-defined map $F_{\db}: X_{\db}\to X_{\db}$ over Besicovitch space.
\end{itemize}\end{defe}
According to \cite{blanchard_cellular_1997}, the Besicovitch space is pathwise-connected, infinite-dimensional and complete, but, it is neither separable nor locally compact.

\subsection{The Feldman space}
Another classical edition distance is the Levenshtein distance \cite{levenshtein1966binary}.
Instead of allowing to edit finite words only via substitution operations (like for the Hamming distance), we now allow to edit using \emph{deletions}.

\begin{defe}~\begin{itemize}
\item The deletion operation $D_j$ at position $j\in \co{0}{\len{u}}$ is defined over word $u\in A^*$ as follows:
$D_j(u)= u_0u_1\ldots u_{j-1}u_{j+1}\ldots u_{\len{u}-1}$, for all $u\in A^*$.
\item The Levenshtein distance $\dll$ is defined over $u,v\in A^*$ as follows:
\[\dll(u,v)=\frac{1}{2}\min\sett{m+m'}{\exists j_1<\cdots <j_m, j'_1<\cdots <j'_{m'}, D_{j_1}\circ\ldots \circ D_{j_m}(u)=D_{j'_1}\circ\ldots \circ D_{j'_{m'}}(v)}.\]
\end{itemize}\end{defe}
Most frequently, we will consider the distance between two words of the same length, so that the result is an integer, and can be defined as the minimal $m$ such that there exist two sequences $D_{j_1}\ldots D_{j_m}$ and $D_{j'_1}\ldots D_{j'_{m}}$ such that $D_{j_1}\circ\ldots \circ D_{j_m}(u)=D_{j'_1}\circ\ldots \circ D_{j'_{m}}(v)$.

The pseudo-metric $\dll(u,v)$ can also be defined as $\frac{\len u+\len v}2-l$, where $l$ is the length of the longest common subword between $u$ and $v$.

Several variants exist in the literature:
\begin{itemize}
	\item One may want to remove factor $\frac12$ in the definition, to make the definition look more natural.
	Nevertheless, the two points above, as well as the next two remarks, motivate our definition.
	Anyway, the two pseudo-metric $\dll$ and $2\dll$ are equivalent.
	\item If one allows two edition operations, insertion and deletion, the purpose could be that it can be defined by performing all operations only on one of the two words.
	The two pseudo-metrics are here exactly equal because an insertion on one side corresponds to a deletion on the other side.
	Manipulations are a little more technical because one has to deal with as many insertion operations as there are letters in the alphabet.
	\item If one additionnally allows the substitution operation from Definition~\ref{d:hamming}, with weight $1$, then again the two obtained pseudo-metrics are equal, because a substitution corresponds to a sequence of an insertion and a deletion.
	\item If one gives the same weights to the substition and deletion operations, then one gets an equivalent pseudo-metric (bounded between $\dll$ and $2\dll$).
\end{itemize}	

\begin{ex}
Let $A=\{0,1\}$.
\begin{enumerate}
\item For $u=010101$ and $v=101010$, we have : $\dll(u,v)=1$. 
\\ Indeed, $D_0(u)=10101$ (we delete the letter of index $0$ in $u$), then we delete the last letter in the end of the word $v$ and we find $D_0(u)=10101=D_5(v)$.
For the sake of comparison, note that $\dhh(u,v)=5$.
\item For $u=0000$ and $v=00001$, we have $\dll(u,v)=\frac{1}{2}$ since it is enough to delete the last lettre of $v$.
\end{enumerate}
\end{ex}
\begin{rem}\label{liv}
For every $u,v\in A^*$, we have:
\[\frac{\abs{\len u-\len v}}2\le \dll(u,v)\le \frac{\len u +\len v}2.\] 
\end{rem}
\begin{proof}
The upper bound comes from the trivial edition sequence producing:
\[D_1\circ D_2\cdots D_{\len{u}}=\lambda=D_1\circ\ldots D_{\len{v}}.\]
On the other hand, if \[D_{j_1}\circ D_{j_2}\circ \cdots \circ D_{j_m}(u)=D_{j'_1}\circ D_{j'_2}\circ \cdots \circ D_{j'_{m'}}(v),\] then \[\len{D_{j_1}\circ D_{j_2}\circ \cdots \circ D_{j_m}(u)}=\len{D_{j'_1}\circ D_{j'_2}\circ \cdots \circ D_{j'_{m'}}(v)}.\]
Hence, $\len u-m=\len v-m'$.
Then we can conclude that \[\frac{\abs{\len{u}-\len{v}}}{2}=\frac{\abs{m-m'}}2\le\frac{m+m'}2=\dll(u,v).\]
\end{proof}
\begin{rem}\label{r:hamlev}
The Hamming distance is an upper bound for the Levenshtein distance, \ie for all words $u,v\in A^*$ such that $\len u=\len v$, \[\dll(u,v)\leq\dhh(u,v).\]
\end{rem}
\begin{proof}
 Let $\dhh(u,v)=m$, then there exists $j_1<\cdots<j_m$ such that for all $h\in \cc1{m}$, $u_{j_h}\neq v_{j_h}$.
 If we delete $u_{j_h}$ and $v_{j_h}$ for all $h\in \cc1{m}$, then we find $ D_{j_1}\circ\ldots \circ D_{j_m}(u)=D_{j_1}\circ\ldots \circ D_{j_m}(v).$ Hence: 
$$\dll(u,v)\leq \frac{2m}{2}=n=\dhh(u,v).\popQED$$
\end{proof}
\begin{pro}\label{p:subadd}
The Levenshtein distance is subadditive, \ie for all words $u,v,u',v'$,
 \[\dll(uu',vv')\leq \dll(u,v)+\dll(u',v').\]
\end{pro}
\begin{proof}
Consider words $u,u',v,v'$, and $m,m',n,n'$ such that:
\[D_{j_1}\circ \cdots \circ D_{j_m}(u)=D_{j'_1}\circ \cdots \circ D_{j'_{m'}}(v)\text{ and }D_{j_{m+1}}\circ \cdots \circ D_{j'_{m+n}}(u')=D_{j'_{m'+1}}\circ \cdots \circ D_{j'_{m'+n'}}(v'),\]
for some minimal edition sequences $j_1<\cdots<j_{m}<\len u$, $j'_{1}<\cdots<j_{m'}<\len{u'}$, $j_{m+1}<\cdots<j_{m+n}\len v$ and $j'_{m'+1}<\cdots<j'_{m'+n'}<\len{v'}$, so that $\dll(u,u')=\frac{m+m'}2$ and $\dll(v,v')=\frac{n+n'}2$.
By concatenating the two previous edited words, we obtain:
\[D_{j_1}\circ \cdots \circ D_{j_{m+n}}(uu')=D_{j'_1}\circ \cdots \circ D_{j'_{m'+n'}}(vv').\]
Therefore $\dll(uu',vv')\leq \frac{m+n+m'+n'}{2}=\dll(u,v)+\dll(u',v')$.
\end{proof}

Following the idea behind the Besicovitch pseudo-metric, we define a pseudo-metric associated to the Levenshtein distance as follows:
\begin{defe}
The \emph{ Feldman} pseudo-metric associated to the Levenshtein distance is: 
$$\dl(x,y)=\limsup_{l\to \infty} \frac{\dll(x_{\co{0}{l}},y_{\co{0}{l}})}{l}, \forall x,y\in A^\N.$$
\end{defe}
Like the Besicovith pseudo-metric, it is a pseudo-metric but not a distance.
Actually all pairs at Besicovitch pseudo-metric $0$ are at Feldman pseudo-metric $0$.
More generally, Remark~\ref{r:hamlev} gives the following.
\begin{rem}
For all $x,y\in A^\N$ we have : $\dl(x,y)\leq \db(x,y).$
\end{rem}

Here too, it is natural to quotient the space of infinite words by the equivalence of zero distance; we obtain a metric space, called the \emph{Feldman space} denoted by $X_{\dl}$.

\section{Dill maps in the Besicovitch space}\label{s:bes}
\subsection{Lipschitz property of dill maps}
It is known since \cite{blanchard_cellular_1997} that every cellular automaton induces a (well-defined) Lipshitz function over Besicovitch space.
\cite[Theorem~13]{muller_curtishedlundlyndon_2009} goes further, by establishing a characterization à la Curtis-Hedlund-Lyndon of cellular automata in the Besicovitch space by three conditions: shift invariance, a condition in terms of uniform continuity and a condition in terms of periodic infinite words.

Some dill maps, on the contrary are not well-defined.
\begin{ex}\label{x:fibobes}
	The Fibonacci substitution is not well-defined over the Besicovitch space $X_\db$.\\
	Even worse, for every $x\in\{0,1\}^\N$ such that $\tilde x\ne\widetilde{\uinf1}$, altering simply the first letter will induce a shift in the substitution.
	Indeed, if $x_0=1$, then:
	\begin{align*}
		\db(\overline\tau(S^0_0(x)),\overline\tau(x))&=\db(0\tau(x),\tau(x)).
	\end{align*}
	Symmetrically, if $x_0=0$, then
	\begin{align*}
	\db(\overline\tau(S^1_0(x)),\overline\tau(x))&=\db(\sigma(\tau(x)),\tau(x)).
\end{align*}
	In both cases, the pseudo-metric	is at least half of the frequence $\db(\uinf1,x)$ of $0$s in $x$.
	For example, $\db(\uinf0,1\uinf0)=0$ but $\db(\overline{\tau}(\uinf0),\overline{\tau}(1\uinf0))=\db(\uinf{(01)},\uinf{(10)})=1$.
	On the other hand, for all $x\in\{0,1\}^\N$, $\db(\overline\tau(x),\overline\tau(\uinf1))\le\db(x,\uinf1)$ (frequence of $0$s in $x$).
	This is still a quite poor continuity set.
\end{ex}

In this section, we characterize dill maps which induce a well-defined function over this space.

Let us note $\maxd=\max\sett{\dhh(f(u),f(v))}{u,v\in A^\diam}$, and $\mind=\min\sett{\dhh(f(u),f(v))}{u\ne v\in A^\diam}$.
\begin{pro}\label{p:subbesi}
Let $F$ be any uniform dill map with diameter $\diam$ and local rule $f$. Then:
\[\forall x,y\in A^\N,\frac{\diam\mind}{\minf{f}}\times \db(x,y)\leq \db(F(x),F(y))\leq \frac{\diam \maxd}{\minf{f}}\times \db(x,y).\]
\end{pro}
\begin{proof}
Let $x,y\in A^\N$ and $l\in \N$. It is clear that we can write $l=|f|m+r$ such that $m\in\N$ and $r<\minf f$, then : 
\begin{eqnarray*}
F(x)_{\co{0}{l}}&=& f(x_{\co{0}{\diam}})f(x_{\co{1}{1+\diam}})\ldots f(x_{\co{m}{m+\diam}})_{\co{0}{r}}, \text{ and }\\ \quad F(y)_{\co{0}{l}}&=& f(y_{\co{0}{\diam}})f(y_{\co{1}{1+\diam}})\ldots f(y_{\co{m}{m+\diam}})_{\co{0}{r}} .
\end{eqnarray*}
Hence if we note $d=\dhh(F(x)_{\co{0}{l}},F(y)_{\co{0}{l}})$, then:
\begin{eqnarray*}
 \sum_{i=0}^{m}\left( \dhh(f(x_{\co{i}{i+\diam}}),f(y_{\co{i}{i+\diam}}))\right) -r \leq & d &\leq  \sum_{i=0}^{m} \dhh(f(x_{\co{i}{i+\diam}}),f(y_{\co{i}{i+\diam}})) \\ 
\sum_{\begin{subarray}{c} i=0 \\ x_{\co{i}{i+\diam}}\neq y_{\co{i}{i+\diam}}\end{subarray}}^{m} \left(\dhh(f(x_{\co{i}{i+\diam}}),f(y_{\co{i}{i+\diam}}))\right) -\minf f  < & d & \leq \sum_{\begin{subarray}{c} i=0 \\ x_{\co{i}{i+\diam}}\neq y_{\co{i}{i+\diam}}\end{subarray}}^{m} \dhh(f(x_{\co{i}{i+\diam}}),f(y_{\co{i}{i+\diam}})) \\ 
 \sum_{\begin{subarray}{c} i=0 \\ \exists j\in \co{i}{i+\diam},  x_j\neq y_j\end{subarray}}^{m}\mind-\minf f < & d & \leq \sum_{\begin{subarray}{c} i=0 \\ \exists j\in \co{i}{i+\diam},  x_j\neq y_j\end{subarray}}^{m} \maxd\\
\sum_{\begin{subarray}{c}  j\in \co{\diam}{m+\diam}\\  x_j\neq y_j\end{subarray}}\sum_{i\in 
\oc{j-\diam }{j}}\mind-\minf f < & d &\leq  \sum_{\begin{subarray}{c}  j\in \co{0}{m+\diam}\\  x_j\neq y_j\end{subarray}}\sum_{i\in \oc{j-\diam }{j}} \maxd \\
\sum_{\begin{subarray}{c}  j\in \co{\delta}{m+\diam}\\  x_j\neq y_j\end{subarray}} \diam\mind-\minf f < & d &\leq  \sum_{\begin{subarray}{c}  j\in \co{0}{m+\diam}\\  x_j\neq y_j\end{subarray}} \diam  \maxd \\
\dhh(x_{\co{\diam}{m+\diam}},y_{\co{\diam}{m+\diam}}) \diam\mind -\minf f< & d &\leq \dhh(x_{\co{0}{m+\diam}},y_{\co{0}{m+\diam}})\diam\maxd\\
 \dhh(x_{\co{0}{m}},y_{\co{0}{m}})\diam\mind-\diam^2\mind -\minf f<& d & \leq \dhh(x_{\co{0}{m}},y_{\co{0}{m}})\diam\maxd+\diam^2\maxd.
\end{eqnarray*}
Hence: 
\begin{eqnarray*}		
\frac{ \dhh(x_{\co{0}{m}},y_{\co{0}{m}})\diam\mind-\diam^2\mind -\minf f}{l}<& \dfrac{d}{l} &\leq  \dfrac{\dhh (x_{\co{0}{m}},y_{\co{0}{m}})\diam\maxd+\diam^2\maxd}{l} \\ 
\dfrac{\diam\mind}{\len{f}} \times\dfrac{\dhh (x_{\co{0}{m}},y_{\co{0}{m}})-\diam^2\mind -\minf f}{m+1}< &\dfrac{d}{l} &\leq \dfrac{\diam\maxd}{\len{f}} \times\dfrac{\dhh (x_{\co{0}{m}},y_{\co{0}{m}})+\diam^2\maxd}{m}\\ 
\dfrac{\diam\mind}{\len{f}} \times\dfrac{\dhh (x_{\co{0}{m+1}},y_{\co{0}{m+1}})-\diam^2\mind -\minf f-1}{m+1}< &\dfrac{d}{l} &\leq \dfrac{\diam\maxd }{\len{f}} \times\dfrac{\dhh (x_{\co{0}{m}},y_{\co{0}{m}})+\diam^2\maxd}{m}.		
\end{eqnarray*}
Finaly, since $m\to \infty$ when $l\to \infty$, then we find: 
\[\frac{\diam\mind}{\minf{f}}\times \db(x,y)\leq \db(F(x),F(y))\leq \frac{\diam\maxd}{\minf{f}}\times \db(x,y).\popQED
\]
\end{proof}
In particular, the Cantor and Thue-Morse substitutions are well-defined over this space (we will discuss them in the next subsection).

We now reach necessary and sufficient conditions for dill maps to induce well-defined dynamical systems over this space. 
\begin{teo}\label{dill Besi}
Let $F$ be a dill map with diameter $\diam$ and local rule $f$. Then the following statements are equivalents: 
\begin{enumerate}
\item\label{i:def} $F_{\db}$ is well-defined.
\item\label{i:lip} $F$ is $\frac{\diam\maxd}{\minf f}$-Lipschitz with respect to $\db$. 
\item\label{i:cst} $F$ is either constant or uniform.
\end{enumerate}
\end{teo}
\begin{proof}~
\begin{itemize}
\item[\ref{i:cst}$\implies$\ref{i:lip}] It is clear from Proposition~\ref{p:subbesi}.
\item[\ref{i:lip}$\implies$\ref{i:def}] It is clear from the definition of Lipschitz function.
\item[\ref{i:def}$\implies$\ref{i:cst}]
	Assume that $\tau$ is nonuniform, \ie there are two words $u$ and $v$ such that $\len{f^*(u)}\ne\len{f^*(v)}$.
	One can assume that their longest common suffix has length $\delta-1$.
	Indeed, otherwise let $a\in A$, $u'=u_{\co{\len u-\diam+1}{\len u}}a^{\diam-1}$ and $v'=v_{\co{\len u-\diam+1}{\len v-1}}a^{\diam-1-l}$; one can note that $f^*(ua^{\diam-1})=f^*(u)f^*(u')$ and $f^*(va^{\diam-1})=f^*(v)f^*(v')$, so that either $\len{f^*(ua^{\diam-1})}\ne\len{f^*(va^{\diam-1})}$, or $\len{f^*(u')}=\len{f^*(ua^{\diam-1})}-\len{f^*(u)}\ne\len{f^*(va^{\diam-1})}-\len{f^*(v)}=\len{f^*(v')}$, and both these pairs of words share a common suffix of length at least $\diam-1$.
	Assume without loss of generality that $k=\len{f^*(u)}-\len{f^*(v)}>0$.
	\begin{itemize}
		\item First assume that there exist $w\in A^*$ and $i\in \N$ such that $f^*(w)_i\neq f^*(w)_{i+k}$.
		 By our previous assumption, we know that $w'=u_{\co{\len u-\diam}{\len u}}\uinf w_{\co0\diam}=v_{\co{\len v-\diam}{\len v-1}}\uinf w_{\co0\diam}$.
		 $F(u\uinf w)=f^*(u)f^*(w')F(\uinf w)$, so that for every $j\in\N$, $F(u\uinf w)_{\len{f^*(u)}+\len{f^*(w')}+j\len{f^*(\uinf w_{\co0{\len w+\diam}})}+i}=w_i$.
		 On the other hand, $F(v\uinf w)_{\len{f^*(u)}+\len{f^*(w')}+j\len{\uinf w_{\co0{\len w+\diam}}}+i}=F(v\uinf w)_{\len{f^*(v)}+k+\len{f^*(w')}+j\len{\uinf w_{\co0{\len w+\diam}}}+i}=w_{i+k}\ne w_i$.
		 We deduce that $\db(F(u\uinf w),F(v\uinf w))\ge\frac1{\len{\uinf w_{\co0{\len w+\diam}}}}$.
		 Since $\len{u}=\len{v}$, we know $\db(u\uinf w,v\uinf w)=0$, so that $F$ is not well-defined over the quotient space.
		\item Otherwise, for all $w\in A^*$, $i\in \co{0}{\len{w}}$, we have $f^*(w)_i=f^*(w)_{i+k}$.
			Let $w'\in A^*$, such that $\len{w'}\geq k$ and let $w\in A^*$. Then:
		\[\forall j\in\N,f^*(w)_j=f^*(w'w)_{j+\len{f^*(w')}}=f^*(w'w)_{j+\len{f^*(w')}\bmod k}=f^*(w')_{j+\len{f^*(w')}\bmod k}.\]
		Hence for all $x\in A^\N$, $j\in \N$, we have $F(x)_j=f^*(w')_{j+\len{f^*(w')}\bmod k}$.
		So, $F$ is constant.
	\popQED
\end{itemize}
\end{itemize}
\end{proof}
\begin{cor}
A substitution $\tau$ yields a well-defined dynamical system $\overline{\tau}$ over $X_{\db}$ if and only if is $1$-Lipschitz with respect to $\db$.
\end{cor}
\subsection{Equicontinuity}
The following derives directly from Proposition~\ref{p:subbesi} (and completeness of the Besicovitch space).
\begin{cor}\label{equi besico}
	Let $F$ be a uniform dill map with diameter $\diam$ and local rule $f$.
	\begin{enumerate}
		\item If $\diam \maxd\leq \minf f$, then $F_{\db}$ is equicontinuous.\\
			For example, for every substitution $\tau$, $\overline\tau$ is equicontinuous.
		\item If $\diam \maxd< \minf f$, then $F_{\db}$ is contracting: every orbit converges to a unique fixed point.
			For example, for every substitution $\tau$ such that $\maxd[\tau]<\minf{\tau}$, $\overline\tau$ is contracting.
		\item If $\diam\mind=\minf f$, then $F_\db$ is an isometry.\\
		 For example, for every substitution $\tau$ such that $\mind[\tau]=\minf\tau$ (which means that the substitution is everywhere marked: any two images have no letter in common), $\overline\tau$ is an isometry.
	\end{enumerate}
\end{cor}

\begin{ex}\label{x:toeplitz}
	Let $\tau$ be a T\oe plitz substitution.
	By definition and by Corollary~\ref{equi besico}, $\overline\tau$ is contracting, so that all orbits converge towards a unique fixed point: the class for $\sim_\db$ of the usual fixed points of the substitution (which is unique if $\forall a,b\in A,\tau(a)_0=\tau(b)_0$, but may not be otherwise, like for the Cantor substitution).
\end{ex}

\begin{ex}\label{x:thuemorse}
	On the contrary, the Thue-Morse substitution is an isometry, thanks to Corollary~\ref{equi besico}.
	In particular, if $\Sigma_\tau$ is the orbit closure of the two fixed points, then for every $x\notin\Sigma_\tau$, the pseudo-metric $\db(\tau^t(x),\Sigma_\tau)$ is constantly positive, so that our intuition that orbits converge towards $\Sigma_\tau$, though justified in the Cantor space, is completely false in the Besicovitch space.
\end{ex}

\begin{rem}
	The behaviors from Example~\ref{x:toeplitz} and~\ref{x:thuemorse} give an essentially full picture of what can occur.
	Indeed, if there exists $p\in\N$ such that $\maxd[\tau^p]<\minf\tau^p$, then $\tau^p$ is contracting; consequently, for every $t\in\N$ the diameter of $\tau^t(A^\N)$ is bounded by that of $\tau^{p\ceil{\frac tp}}(A^\N)$, which is bounded by $\left(\frac{\maxd[\tau^p]}{\minf\tau^p}\right)^{\ceil{\frac tp}}$; so all orbits of $\tau$ converge towards a unique fixed point.\\
	If, on the other hand, for every $t\in\N$, $\maxd[\tau^t]=\minf\tau$, this means that there exists a subalphabet $A_t\in A$ containing at least two letters, such that $a,b\in A_t\implies\dhh(\tau^t(a),\tau^t(b))=\minf\tau^t$.
	It is not difficult to see that $A_t\subset A_{t+1}$, and since it is finite, the subalphabet $A'=\bigcap_{t\in\N}A_t$ contains at least two letters.	
	Then the restriction of $\overline\tau$ to ${A'}^\N$ is an isometry (because Proposition~\ref{p:subbesi} remains true when the minimum and maximum are taken over a subalphabet).
\end{rem}

The links between dynamical properties in the Cantor space and in the Besicovitch space appeared for cellular automata in \cite{blanchard_cellular_1997}, \cite{formenti_dynamics_2009}: in particular, sensitivity in the Besicovitch space implies sensitivity in the Cantor space, or equicontinuity in the Cantor space implies equicontinuity in the Besicovitch space.
Nevertheless, unlike for cellular automata, there exist dill maps which are equicontinuous in the Cantor space but not in the Besicovitch space. 
\begin{ex}
Consider the dill map $F$ with diameter $2$ defined as the composition of the doubling substitution and of the Min CA, or equivalently by the following local rule:
	\begin{eqnarray*}
		f : 00,01,10 &\mapsto & 00 \\
		11 & \mapsto & 11.
	\end{eqnarray*}
	This dill map is $1$-Lipschitz in the Cantor space, but $F_{\db}$ is not equicontinuous.
	We can even prove that $\uinf1$ is $1^-$-unstable.
	For $p\in\Ns$, let $y=\uinf{(1^{p-1}0)}$.
	Remark that $\db(x,y)=\frac{1}{p}$. 
	On the other hand, $F^{p-1}(y)=\uinf0$, so that $\db(F^{p-1}(\uinf1),F^{p-1}(y))=\db(\uinf1,\uinf0)=1$. 
\end{ex}
Some weak robustness properties of cellular automata from \cite{formenti_dynamics_2009}, though, can be generalized to dill maps, like in the following statement.
\begin{pro}
	Let $F$ be a uniform dill map and $m\in \N$. Then we have the following:
	\begin{enumerate}
		\item If $F_{\db}$ is sensitive, then $\sigma_{\db}^m\circ F_{\db}$ is sensitive, too.
		\item If $x$ is an equicontinuous point for $F_{\db}$, then it is an equicontinuous point for $\sigma_{\db}^m\circ F_{\db}$, too.
		\item If $F_{\db}$ is equicontinnuous, then $\sigma_{\db}^m\circ F_{\db}$ is equicontinuous, too.
	\end{enumerate}
\end{pro}
\begin{proof}
	The key to proof the three statements is the following:
	\[\forall n\in \N,(\sigma^m\circ  F)^n=\sigma^{\sum_{k=0}^{n-1}m\minf{f}^k}\circ F^n.\]
	Let us prove this by induction on $n\in\N$.
	The case $n=0$ is obvious.
	Suppose that for it is true for some $n$.
	\begin{eqnarray*}
		(\sigma^m \circ  F)^k &=&(\sigma^{m}\circ F)\circ(\sigma^m\circ  F)^n \\
		&=& \sigma^{m}\circ F \circ (\sigma^{\sum_{k=0}^{n-1}m\minf{f}^k}\circ F^n) \\
		&=& \sigma^{m}\circ \sigma^{\minf{f}\sum_{k=0}^{n-1}m\minf{f}^k}\circ F \circ F^n,
	\end{eqnarray*}
	which is the next step of the induction hypothesis.
	The last equality comes from the fact that $F$ is uniform, so that $\cocy[n]{x}=\minf{f}=\maxf{f}$.
	Now we can deduce the proof of the statement: for all $x,y\in A^\N$ and for all $ n,m\in \N$: 
	\begin{eqnarray*}
		\db((\sigma^m\circ F)^n(x),(\sigma^m\circ F)^n(y)) &=&\db(\sigma^{\sum_{k=0}^{n-1}m\minf{f}^k}( F^n(x)),\sigma^{\sum_{k=0}^{n-1}m\minf{f}^k}( F^n(y))) \\&=& \db( F^n(x),F^n(y)) 
		\quad\text{(since $\db$ is shift-invariant)}
		\popQED
	\end{eqnarray*}
\end{proof}

It was known that the cellular automata suit well in the Besicovitch pseudo-metric, and we have seen in this section that it is also the case of uniform substitutions.
But we proved that this is not true for non-uniform substitutions.
In the next section, we consider another topological space, in which both cellular automata and substitutions are well-defined over this space.
\section{Dill maps in the Feldman space}\label{s:feld}
\subsection{Shift}
One of the motivation to study the Besicovitch space is that the shift is an isometry over this space.
In the Feldman space, this is still true, but even more than this: the shift is exactly the indentity.
\begin{pro}\label{p:cashift}
The shift over the Feldman space is the identity map.
\end{pro}
\begin{proof}
Let $\tilde{x}\in X_{\dl}$. If $x\in \tilde{x}$, then
$\dll(x_{\co{0}{l}},\sigma(x)_{\co{0}{l}})=\dll(x_{\co{0}{l}},x_{\co{1}{1+l}})\le1$:
simply delete the first letter of $x_{\co{0}{l}}$ and the last letter of $\sigma(x)_{\co{0}{l}}$, to obtain $x_{\co{1}{l}}$ in both cases.
Hence : 
\[\dl(x,\sigma(x))=\limsup_{l\to\infty} \dfrac{\dll(x_{\co{0}{l}},x_{\co{1}{1+l}})}{l} \leq \limsup_{l\to\infty} \dfrac{1}{l}=0.\popQED\]
\end{proof}
Since every equivalence class is invariant by shift, dynamical systems over this space can be considered as acting on shift orbits.
\begin{lem}\label{l:plage}
For every $x,y\in A^\N$, if $\dl(x,y)=0$ then for all $p\in \N,$ there exists $i,j\in \N$ such that : $$x_{\co{i}{i+p}}=y_{\co{j}{j+p}}.$$
\end{lem}
\begin{proof}
Let $x,y\in A^\N$ such that $\dl(x,y)=0$ and $p\in \N$. Then : $$\forall \varepsilon>0,\exists N>0,\forall n>N, \dll(x_{\co{0}{n}},y_{\co{0}{n}})<n\times \varepsilon.$$
Hence, for $\varepsilon=\dfrac{1}{p+1}$, there exists $N_1>0$ such that for all $n>N_1$ we have $\dll(x_{\co{0}{n}},y_{\co{0}{n}})<\dfrac{n}{p+1}$.
Then for $\varepsilon=\dfrac{1}{p+1}$, there exists $N=\max\{p+1,N_1\}$ such that for all $n>N$ we have $\dll(x_{\co{0}{n}},y_{\co{0}{n}})<\dfrac{n}{p+1}$.
Hence, there exists $u\in A^{p}$ such that $u\sqsubseteq x_{\co{0}{n}}$ and $u\sqsubseteq y_{\co{0}{n}}$.
So, there exists $i,j\in\co{0}{n-p}$ such that : $$u=x_{\co{i}{i+p}}=y_{\co{j}{j+p}}.\popQED
$$
\end{proof}
\begin{lem}\label{l:cashift}
Let $F$ be a CA with diameter $\diam$ and local rule $f$. For every $u\in A^\diam$, if $\dl(\uinf u,F(\uinf u))=0$ then there exists $k\in\co0\diam$ such that $F(\uinf u)=\sigma^k(\uinf u)$.
\end{lem}
\begin{proof}
Let $u\in A^\diam$ be such that $\dl(\uinf u,F(\uinf u))=0$. 
According to Lemma~\ref{l:plage}, there exist $i,j\in \N$ such that $F(\uinf u)_{\co{i}{i+\diam}}=(\uinf u)_{\co{j}{j+\diam}}$.
Since both $\uinf u$ and $F(\uinf u)$ are $\diam$-periodic for the shift, one can write, for every $k\in\N$,
\[F(\uinf u)_k=F(\uinf u)_{k-i\bmod\diam+i}=(\uinf u)_{k-i\bmod\diam+j}=(\uinf u)_{j-i+k\bmod\diam}.\]
Hence $F(\uinf u)=\sigma^{j-i+k\bmod\diam}(\uinf u)$.
%
\end{proof}
\begin{teo}
For every CA $F$ we have,  $F\in \widetilde{\id} \iff \exists k\in \N,  F=\sigma^k.$
\end{teo}
\begin{proof}~\begin{itemize}
\item[$"\Leftarrow"$] Proposition~\ref{p:cashift} proves that if $F=\sigma^k$ then $F\in \tilde{\id}$.
\item[$"\Rightarrow"$] Let $F$ be a CA with diameter $\diam\in \Ns$ and local rule $f$ such that $F\in \tilde{\id}$.
Then for all $x\in A^\N$ we have $\dl(F(x),x)=0$.
In particular, for all $u\in A^{\diam}$ we have $\dl(F(\uinf u),\uinf u)=0.$
So according to Lemma~\ref{l:cashift}, we deduce that for all $u\in A^{\diam}$, there exists $k_u\in\co0\diam$ such that 
$f(u)=F(\uinf u)_0=\sigma^{k_u}(\uinf u)_0.$\\
Defining $k$ to be any common multiple of all the ${k_u}$, for ${u\in A^\diam}$, we have:  
\[\forall u\in A^{\diam}, f(u)=\sigma^{k}(\uinf{u})_0=u_{k\bmod\diam}.\]
This is exactly the local rule of $\sigma^{k\bmod\diam}$.
\popQED\end{itemize}\end{proof}
\subsection{Lipschitz property of dill maps}
Now, we aim at proving that, unlike in the Besicovitch space, all dill maps are well-defined in the Feldman space.

\begin{lem}~\label{l:Lev}
	Let $F$ be a dill map and $M,M'\in\N$ such that for every $u\in A^{*}$ and every $j\in\co0{\len u}$, 
	\begin{align*}
	\dll(f^*(D_j(u)),f^*(u))&\le M+\frac{\len{f^*(u)}-\len{f^*{D_j(u)}}}2
	\\\dll(f^*(D_j(u)),f^*(u))&\le M'-\frac{\len{f^*(u)}-\len{f^*{D_j(u)}}}2.
\end{align*}
	Then for all $l\in\N$ and $u,v\in A^l$, we have:
	\[\dll(f^*(u),f^*(v))\leq(M+M')\dll(u,v)-\frac{\abs{\len{f^*(u)}-\len{f^*(v)}}}2.\]
\end{lem}
\begin{proof}
Consider words $u,v$, and $m$ such that:
\[D_{j_1}\circ \cdots \circ D_{j_m}(u)=D_{j'_1}\circ \cdots \circ D_{j'_{m}}(v),\]
for some minimal edition sequences $j_1<\cdots<j_{m}<\len u$ and $j'_{1}<\cdots<j_{m}<\len v$, so that $\dll(u,v)=m$.
By the triangular inequality, one gets:
\begin{align*}
	\dll(f^*(u),f^*(v))&\le\sum_{k=1}^m\dll(f^*(D_{j_{k+1}}\circ\cdots\circ D_{j_m}(u)),f^*(D_{j_k}\circ\cdots\circ D_{j_m}(u)))+\\
	&\quad+\dll(f^*(D_{j_1}\circ\cdots \circ D_{j_m}(u)),f^*(D_{j_1}\circ\cdots\circ D_{j_m}(v)))+ \\
	&\quad+\sum_{k=1}^{m}\dll(f^*(D_{j'_{k}}\circ\cdots\circ D_{j'_m}(v)),f^*(D_{j'_{k+1}}\circ\cdots\circ D_{j'_m}(v))).
\end{align*}
	Now our two assumptions allow to write:
\begin{align*}
		\dll(f^*(u),f^*(v))&\le \sum_{k=1}^m\left(M+\frac{\len{f^*(D_{j_{k+1}}\circ\cdots\circ D_{j_m}(u))}-\len{f^*(D_{j_{k}}\circ\cdots\circ D_{j_m}(u))}}2\right)+0+\\
			&\quad+ \sum_{k=1}^m\left(M'+\frac{\len{f^*(D_{j'_{k}}\circ\cdots\circ D_{j'_m}(v))}-\len{f^*(D_{j'_{k+1}}\circ\cdots\circ D_{j'_m}(v))}}2\right)\\
			&\le Mm+\frac{\len{f^*(u)}-\len{f^*(D_{j_1}\circ\cdots\circ D_{j_m}(u))}}2+
			M'm+\frac{\len{f^*(D_{j'_1}\circ\cdots\circ D_{j'_m}(v))}-\len{f^*(v)}}2 
		\\&\le (M+M')m+\frac{\len{f^*(u)}-\len{f^*(v)}}2.
\end{align*}
\end{proof}
\begin{lem}\label{l:lipfeld}
	Let $F$ be a dill map and $M,M'\in\N$ such that for every $u\in A^{*}$ and every $j\in\co0{\len u}$, 
	\begin{align*}
	\dll(f^*(D_j(u)),f^*(u))&\le M+\frac{\len{f^*(u)}-\len{f^*{D_j(u)}}}2
\\\dll(f^*(D_j(u)),f^*(u))&\le M'-\frac{\len{f^*(u)}-\len{f^*{D_j(u)}}}2.
\end{align*}
	Let $x$ be such that for every $i\in\N$, $\len{f(x_{\co i{i+\diam}})}\ge L$.
	Then for every $y\in A^\N$,
	\[\dl(F(x),F(y))\le\frac{M+M'}L\dl(x,y).\] 
\end{lem}
\begin{proof}
	Let $x,y\in A^\N$ and $l\in \N$.
	Consider the largest $k\in\N$ such that $\len{f(x_{\co0k})}\le l$.
	Then $F(x)_{\co0l}$ can be written $f(x_{\co0k})w$ for some $w$ of length less than $\maxf f$.
	Note that $l=\sum_{i=0}^{k-\diam}\len{f(x_{\co{i}{i+\diam}})}\ge(k-\diam+1)L$.
	Proposition~\ref{p:subadd} gives the following:
	\begin{align*}\dll(F(x)_{\co{0}{l}},F(y)_{\co{0}{l}})
		&\leq \dll(f^*(x_{\co{0}{k}}),f^*(y_{\co{0}{k}}))+\dll(w,F(y)_{\co{\len{f^*(y_{\co0k})}}l}).
	\end{align*}
		Note that the previous inequality still holds if $\len{f^*(x_{\co0k})}\ge l$, in which case the second term is $\dll(w,\lambda)=\frac{\len w}2$.
		Otherwise,
	\begin{align*}
		\dll(w,F(y)_{\co{\len{f^*(x_{\co0k})}}l})&\le\frac{\len w+\len{F(y)_{\co{\len{f^*(x_{\co0k})}}l}}}2
		\\&\le\frac{\len w+(l-\len{f^*(y_{\co0k})})+(\len{f^*(y_{\co0k})}-\len{f^*(x_{\co0k})})}2.
		\\&<\frac{(\len{f^*(y_{\co0k})}-\len{f^*(x_{\co0k})})}2+\maxf f-1.
	\end{align*}
	We can use Lemma~\ref{l:Lev} to get:
	\begin{align*}\dll(F(x)_{\co{0}{l}},F(y)_{\co{0}{l}})
		&<(M+M')\dll(x_{\co{0}{k}},y_{\co{0}{k}})-\frac{\abs{\len{f^*(x_{\co{0}{k}})}-\len{f^*(y_{\co{0}{k}})}}}2+
		\\&\quad+\frac{(\len{f^*(y_{\co0k})}-\len{f^*(x_{\co0k})})}2+\maxf f-1
		\\&<(M+M')\dll(x_{\co{0}{k}},y_{\co{0}{k}})+\maxf f-1
	\end{align*}
	Since $l\ge(k-\diam+1)L$, one can write:
	\begin{align*}
		\frac{\dll(F(x)_{\co{0}{l}},F(y)_{\co{0}{l}})}{l}
			&<\frac{(M+M')\dll(x_{\co{0}{k}},y_{\co{0}{k}})+\maxf{f}-1}{(k-\diam+1)L}
			\\&\sim_{k\to\infty}\frac{M+M'}L\frac{\dll(x_{\co{0}{k}},y_{\co{0}{k}})}k
		\end{align*}
	Finally since $k$ tends to infinity when $l$ tends to infinity, we have:
	\begin{align*}
		\dl(F(x),F(y))&=\limsup_{l\to\infty}\frac{\dll(F(x)_{\co{0}{l}},F(y)_{\co{0}{l}})}{l}
		\\&\le\limsup_{k\to\infty}\frac{M+M'}L\frac{\dll(x_{\co{0}{k}},y_{\co{0}{k}})}k
		\\&\le\frac{M+M'}L\dl(x,y).
	\end{align*}
\end{proof}

\begin{cor}\label{c:dilllip}~\begin{enumerate}
		\item\label{i:dilllip} Let $F$ be a dill map and $M,M'\in\N$ such that for every $u\in A^{*}$ and every $j\in\co0{\len u}$, 
	\begin{align*}
	\dll(f^*(D_j(u)),f^*(u))&\le M+\frac{\len{f^*(u)}-\len{f^*{D_j(u)}}}2
\\\dll(f^*(D_j(u)),f^*(u))&\le M'-\frac{\len{f^*(u)}-\len{f^*{D_j(u)}}}2.
\end{align*}
		Then $F$ is $\frac{M+M'}{\minf{f}}$-Lipschitz.
		\item\label{i:dillpt} Let $F$ be any dill map with diameter $\diam$, and $x$ be such that for every $i\in\N$, $\len{f(x_{\co i{i+\diam}})}\ge L$.
			Then, for every $y\in A^\N$,
		\[\dl(F(x),F(y))\le({2\diam-1})\frac{\maxf f}L\dl(x,y).\]
	For example, if $f$ is a substitution $\tau$ and $x$ is such that for every $i\in\N$, $\len{\tau(x_i)}\ge L$, then for every $y\in A^\N$,
\[\dl(\overline\tau(x),\overline\tau(y))\le\frac{\maxf\tau}L\dl(x,y).\]
\item In particular, any dill map with diameter $\diam$ and lower norm $\minf f$ is $({2\diam-1})\frac{\maxf f}{\minf f}$-Lipschitz.
\\For example, any substitution $\tau$ yields a $\frac{\maxf\tau}{\minf\tau}$-Lipschitz dynamical system.
\end{enumerate}\end{cor}
\begin{proof}~\begin{enumerate}
\item Since, for all $x\in A^\N$, all $i\in \N$, we have $\len{f(x_{\co{i}{i+\diam}})}\geq \minf{f}$, then according to Lemma~\ref{l:lipfeld}, we find that $F$ is $\frac{M+M'}{\minf{f}}$-Lipschitz.
\item Let $u\in A^{2\diam-1}$. Then according to Remark~\ref{liv} we have: 
\[ \dll(f^*(u),f^*(D_j(u)))\le \frac{\len{f^*(u)}+\len{f^*(D_j(u))}}{2}.\]
Hence, 
\[ \dll(f^*(u),f^*(D_j(u)))\le \len{f^*(u)}-\frac{\len{f^*(u)}-\len{f^*(D_j(u))}}{2}\le \diam-\frac{\len{f^*(u)}-\len{f^*(D_j(u))}}{2}.\]
Also, 
\[ \dll(f^*(u),f^*(D_j(u)))\leq \len{f^*(D_j(u))} -\frac{\len{f^*(D_j(u))-\len{f^*(u)}}}{2}\le (\diam-1)+\frac{\len{f^*(u)}-\len{f^*(D_j(u))}}{2}.\]
The result follows from Lemma~\ref{l:lipfeld}.
\item Since, for all $x\in A^\N$ we have $\len{f(x_{\co{i}{i+\diam}})}\geq \minf{f}$, then we deduce the result from \ref{i:dillpt}.
\end{enumerate}
\end{proof}
		
\subsection{Equicontinuity}
Finally, we study some dynamical properties since the study of some other properties brings us to some unsolved problems of word algorithmics and other with exponential complexity.

We can already derive from Corollary~\ref{equi besico} (since the Besicovitch topology is less fine) that all uniform substitutions yield equicontinuous dynamical systems in the Feldman space.
The following theorem generalizes the result by establishing a characterization of equicontinuous substitutions.
%
\begin{teo}\label{t:primeq}
	Let $\tau$ be any substitution. Consider the dynamical system $\overline\tau$ over the Besicovitch space. Then:
	\begin{enumerate}
		\item The infinite words of ${\maxal\tau}^\N$ are equicontinuous.
		\item\label{i:noneq} The infinite words of $(\compl{\maxal\tau})^\N$ are not equicontinuous.
\end{enumerate}\end{teo}
	The following corollary can directly be derived from the theorem, by noting that $\maxal\tau$ is never empty.
\begin{cor}\label{irre-sub}~\begin{enumerate}
	\item A substitution is equicontinuous if and only if all of its terminal components are maximum.
	\item This is the case for irreducible substitutions, and uniform substitutions.
	\item Every substitution admits at least one equicontinuous infinite word.
\end{enumerate}\end{cor}
\begin{proof}[Proof of Theorem~\ref{t:primeq}]
	 Let $\tau$ be a substitution, $M(\tau)$ its matrix and $\rho_+$ its spectral radius.
	 In particular, $\maxf{\tau^t}\le\alpha\rho_+^t$ for some $\alpha>0$.
	\begin{enumerate}
	\item Let $x\in{\maxal\tau}^\N$, \ie there exists $\beta>0$ such that for every $i\in\N$ and $t\in\N$, $\len{\tau^t(x_i)}>\beta
	\rho_+^t$.
	From Point~\ref{i:dillpt} of Corollary~\ref{c:dilllip}, for every $y\in A^\N$,
	 \begin{align*}
		\dl(\overline\tau^t(x),\overline\tau^t(y))&\le\frac{\maxf{\tau^t}}{\beta\rho_+^t}\dl(x,y)
		\\&\le\frac{\alpha}{\beta}\dl(x,y).
		\end{align*}
	So all iterates $\overline\tau^t$ are Lipschitz with a uniform coefficient: $\overline\tau$ is equicontinuous.
	\item Consider any infinite word $x\in(\compl{\maxal\tau})^\N$.
		By definition of $\maxal\tau$, $\rho_-=\max_{i\in\N}\rho_{x_i}$ is strictly smaller than the spectral radius $\rho_+$ of $M(\tau)$.
		Let $k\in\Ns$, and $y$ defined by $y_i=x_i$ for every $i\notin k\N$, and $y_i$ be any letter from a terminal maximal component, if $i\in k\N$.
		Note that $\dl(x,y)\le\db(x,y)\le\frac1k$.\\
		Besides, if $u\in(\compl{\maxal\tau})^*$, then for every $v\in A^*$ of the same length, $\dll(u,v)\ge\card{\sett i{\co0{\len v}}{v_i\in\maxal\tau}}$, because every edition sequence must delete at least all letters of $\maxal\tau$ from $v$, and as many letters from $u$.
		Also remark that if $a$ is from a terminal maximal component, then $\tau^t(a)$ contains no letter from $\compl{\maxal\tau}$, by definition.
		By combining the previous two arguments, we get that for every $u\in(\compl{\maxal\tau})^*$ and $v\in A^*$ such that $\len u=\len{\tau^t(v)}$, $\dll(u,\tau^t(v))$ is at least the sum, for all $v_i$ that come from a terminal maximal component, of the $\len{\tau^t(v_i)}$.
		In particular, for every $m\in\N$,
		\begin{align*}
			\dl\left(\overline\tau^t(x)_{\co0{\len{\tau^t(y_{\co0{km}})}}},\tau^t(y_{\co0{km}})\right)&\ge\sum_{i=0}^{m-1}\len{\tau^t(y_{ki})}
			\\&\ge\len{\tau^t(y_{\co0{km}})}-\sum_{i=0}^{m-1}\len{\tau^t(y_{\oo{ki}{(k+1)i}})}
			\\&\ge\len{\tau^t(y_{\co0{km}})}\left(1-\frac{1}{1+\frac{\sum_{i=0}^{m-1}\len{\tau^t(y_{ki})}}{\sum_{i=0}^{m-1}\len{\tau^t(y_{\oo{ki}{(k+1)i}})}}}\right).
			\end{align*}
		Now there exist $\alpha,\beta>0$ such that, for $t\in\N$ large enough, we have $\len{\tau^t(a)}<\alpha\rho_-^t$ for every $a\in(\compl{\maxal\tau})^\N$ and $\len{\tau^t(a)}>\beta\rho_+^t$ for every $a\in{\maxal\tau}^\N$.
	After renormalizing, we get:
		\begin{align*}
			\frac{\dl\left(\overline\tau^t(x)_{\co0{\len{\tau^t(y_{\co0{km}})}}},\overline\tau^t(y)_{\co0{\len{\tau^t(y_{\co0{km}})}}}\right)}{\len{\tau^t(y_{\co0{km}})}}
				&\ge 
					1-\frac1{1+\frac\beta{(k-1)\alpha}\left(\frac{\rho_+}{\rho_-}\right)^t}
		\end{align*}
	Overall, we obtain: \[\dl(x,y)\ge\limsup_{m\to\infty}\frac{\dl\left(\overline\tau^t(x)_{\co0{\len{\tau(y_{\co0{km}})}}},\overline\tau^t(y)_{\co0{\len{\tau^t(y_{\co0{km}})}}}\right)}{\len{\tau(y_{\co0{km}})}}=1.\]
		Since $k$ was taken arbitrary, $y$ is arbitrarily close to $x$, so that $x$ is $1^-$-unstable.		
\popQED\end{enumerate}\end{proof}

Apart from the subalphabet argument from Point~\ref{i:noneq} of Theorem~\ref{t:primeq}, it is usually quite hard to prove lower bounds for the Feldman pseudo-metric.
In particular, we have no example of a dill map without equicontinuous infinite word.

\begin{ex}
The Fibonacci substitution $\tau$ is primitive, so $\overline\tau$ is equicontinuous in the Feldman space, though almost no point is equicontinuous in the Besicovitch space (see Example~\ref{x:fibobes}). 
\end{ex}
\begin{ex}
Let $F$ be the Xor $CA$. Then neither $F_{\dl}$ nor $F_{\db}$ is equicontinuous.
Indeed, let us prove that $\uinf0$ is $1^-$-unstable.
Let $k\in\N$, and $y=0^{2k-1}1$.
Then $\dl(x,y)\leq\db(x,y)\leq\frac{1}{2^k}$.
A classical induction on $k$ gives that $F^{2^k-1}(y)=\uinf1$ (for more details see  \cite[Example~5.6]{kurka2003topological}).
Hence: \[\db(F^p(x),F^p(y))=\dl(F^p(x),F^p(y))=1.\]
Figure~\ref{Xor figure} illustrates that $\uinf1$ is also a nonequicontinuous point. 
\begin{figure}[H]
\begin{center}
\includegraphics[scale=0.2]{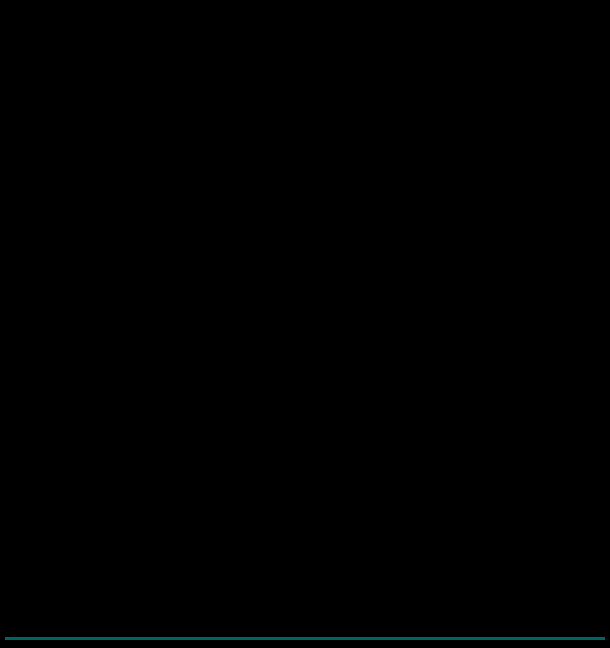}~~~~~~~~
\includegraphics[scale=0.2]{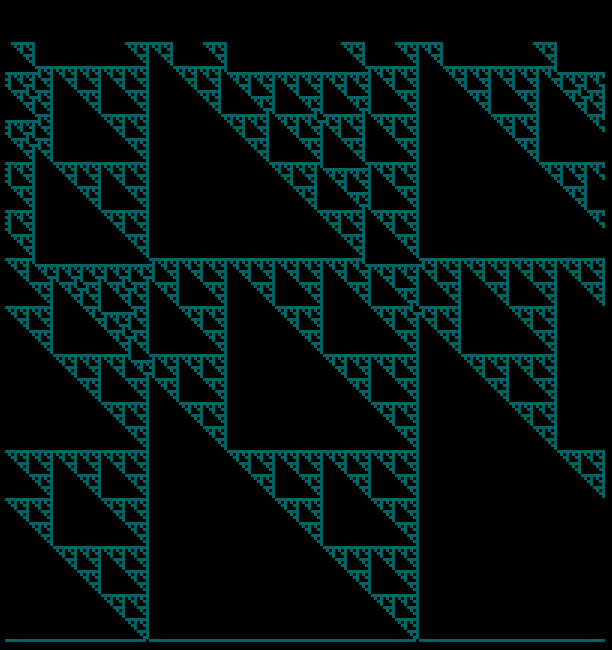}
\end{center}
\caption{A non-equicontinuous point with respect to $\dl$ and $\db$.}
\label{Xor figure}
\end{figure}
\end{ex}
\begin{ex}
Let $\tau$ be a substitution defined as follows :
\begin{eqnarray*}
\tau : 0 &\mapsto & 1 \\
1 & \mapsto & 00.
\end{eqnarray*}
Then, $M(\tau)=
\begin{bmatrix}
0&1\\
2&0
\end{bmatrix}$ and $M(\tau)^2=
\begin{bmatrix}
2&0\\
0&2
\end{bmatrix}$.
Hence, $M(\tau)$ is irreducible. So, according to Corollary~\ref{irre-sub}, we deduce that $\overline{\tau}_{\dl}$ is equicontinuous.
More precisely, $\tau^2$ is actually the doubling substitution, proven to be $1$-Lipschitz by Corollary~\ref{c:dilllip}.
Besides, it can be shown that the latter is even an isometry: the longest common subword of $\tau^2(u)$ and $\tau^2(v)$ is always obtained from doubling a common subword of $u$ and $v$.
\end{ex}
Though primitive or uniform substitutions behave smoothly in our paces, other substitutions may be more pathological.
\begin{ex}
Let $\tau$ be a substitution defined over $A=\{ 0,1\}^\N$ as follows :
\begin{eqnarray*}
\tau ~:~ 0 & \mapsto & 0 \\
1  &\mapsto & 11. 
\end{eqnarray*}
We can remark that, 
$M(\tau)=\begin{bmatrix}
1 & 0\\
0 & 2
\end{bmatrix}$, 
and, $M^n(\tau)=\begin{bmatrix}
1 & 0\\
0 & 2^n
\end{bmatrix}$.
Hence there are two compenents, $\{0\}$ and $\{1\}$, the latter being the the maximal compenent, so $\maxal\tau=\{1\}$. Then, according to Theorem~\ref{t:primeq}, $\uinf1$ is an equicontinuous point of $\overline{\tau}_{\dl}$ and $\uinf0$ is not, as illustrated in Figures~\ref{f:neqt} and~\ref{f:eqt} (where $1$ is represented in black, and $0$ in red).
\begin{figure}[H]\label{f:eqt}
\begin{center}
\includegraphics[scale=0.35]{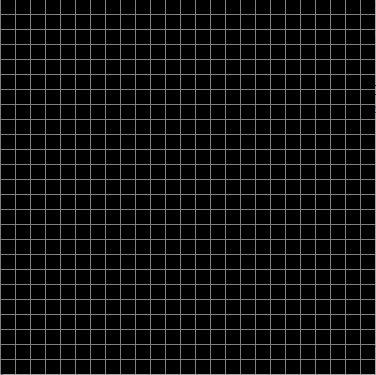}\qquad
\includegraphics[scale=0.35]{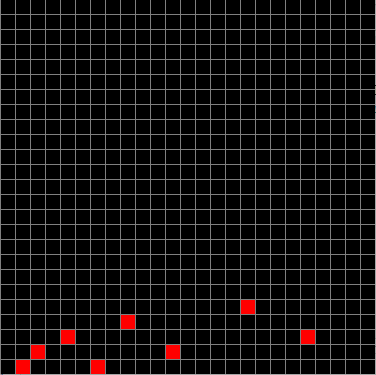}
\caption{$\uinf1$ is equicontinuous point for $\overline{\tau}$.}
\end{center}
\end{figure}
\begin{figure}[H]\label{f:neqt}
\begin{center}
\includegraphics[scale=0.35]{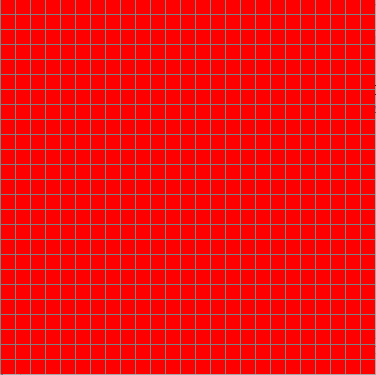}\qquad
\includegraphics[scale=0.35]{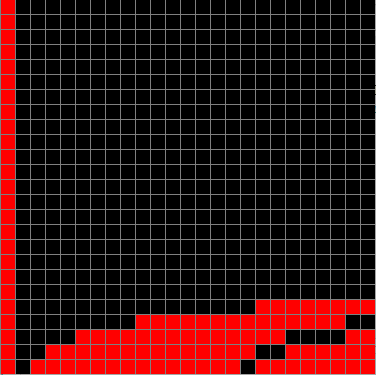}
\caption{$\uinf0$ is a non-equicontinuous point for $\overline{\tau}$.}
\end{center}
\end{figure}
\end{ex}

We have seen in the Besicovitch space (see \ref{x:thuemorse}) that the orbits a dill map may not converge towards its classical limit set.
We conjecture that this does not occur in the Feldman space.
\cite{garcia2020topological,ornstein_donald_s._equivalence_1982}, prove together, among other results, that the limit set of a primitive substitution is a singleton in the Feldman space.
The tools involved may be useful to understand our question.

\subsection{Expansiveness}
For the case of Besicovitch space, it is already proved in \cite{blanchard_cellular_1997} that there is no expansive CA over this space.
By the same method, we prove that there is no expansive CA over the Feldman space.
{Note that it does not derive directly from the corresponding result in the Besicovitch space, because some infinite words could be in the same Feldman class, hence not concerned by the expansiveness property, but not in the same Besicovitch class.}
\begin{teo}
	There is no expansive CA in the Feldman space.
\end{teo}
\begin{proof}
	Let $F$ be a CA with diameter $\diam$ and local rule $f$.
	Let $\varepsilon>0$ and $p\in \N$ such that $\frac{1}{1+p}<\varepsilon$. 
	To prove that $F$ is not expansive it is enough to find two infinite words at nonzero Feldman pseudo-metric such that $\dll(F^n(z),F^n(z'))<\varepsilon, \forall n\in \N$.
	From the proof of the case of Besicovitch space \cite{blanchard_cellular_1997} we can take $x,y\in A^\N$ such that :
	\begin{eqnarray*}
		x & = & 0^{n^1}1^{n^0}0^{n^3}1^{n^2}\cdots .\\
		y & = & 1^{n^1}0^{n^0}1^{n^3}0^{n^2}\cdots .
	\end{eqnarray*}
	We now distinguish four cases, depending on the image of $0^{\diam}$ and $1^{\diam}$.
	\begin{enumerate}
		\item The first case is when $f(0^{\diam} )=f(1^{\diam})=0$. We find $\dl(F(x),0^\infty )\leq \db(F(x),0^\infty )=0$, then we take $z=0^\infty$ and $z'=x$, since :  
		$$\dl(F^t(0^\infty),F^t(x))\leq \db (F^t(0^\infty),F^t(x))<\varepsilon ,\forall t\in \N.$$
		\item If $f(0^{\diam})=0$ and $f(1^{\diam})=1$ we find $\dl(F(x),x)\leq\db(F(x),x)=0$, then:  
		$$\dl (F^t(0^\infty),F^t(x))\leq \db(F^t(0^\infty),F^t(x))=\dfrac{1}{n+1}<\varepsilon,  \forall t\in \N.$$
		\item If $f(0^{\diam})=1$ and $f(1^{\diam})=1$ we find $\dl(F(y),1^\infty)\db(F(y),1^\infty)=0$, then:  
		$$\dl(F^t(1^\infty),F^t(y))\leq \db(F^t(1^\infty),F^t(y))<\varepsilon , \forall t\in \N.$$
		\item If $f(0^{\diam})=1$ and $f(1^{\diam})=0$ then :
		\begin{itemize}
			\item $\dl(1^\infty,F(0^\infty))\leq\db(1^\infty,F(0^\infty))=0$ and $\dl(F(1^\infty),0^\infty)\leq \db (F(1^\infty),0^\infty)=0$.
			\item $\dl(F(x),y)\leq\db(F(x),y)=0$ and $\dl(x,F(y))\leq\dl(x,F(y))=0$.
		\end{itemize}
		Hence: 
		$$\dl (F^t (0^\infty ),F^t (x))\leq \db(F^t (0^\infty ),F^t (x))=\dfrac{1}{n+1}<\varepsilon , \forall t\in\N.$$
	\end{enumerate}
	Then we found two pairs of infinite words $(x,0^\infty)$ and $(y,1^\infty)$ such that: 
	$$\dl (F^t (0^\infty ),F^t (x))\leq \varepsilon, \text{ and } \dl (F^t (1^\infty ),F^t (y))\leq \varepsilon.$$
	On the other hand, we can remark that $\dl(x,0^\infty)>0$ and $\dl(y,1^\infty)>0$.
	Hence, $F$ is not expansive.
\end{proof}
\section{Conclusion and perspective}
In this paper, we study CA, substitutions and in general dill maps over two non trivial topological spaces (Besicovitch space and Feldman space). 
Those spaces were constructed using two pseudo-metrics depending on two different edit distances over finite words (Hamming distance and Levenshtein distance). 
Over the Feldman space, the shift is equal to the identity, there are no expansive CA, every substitution is well-defined and admits at least one equicontinuous point.
This topology turns out to be a suitable playground for the study of the dynamical behavior of dill maps.
The construction of the Feldman pseudo-metric was made only by changing the Hamming distance with another edit distance.
This makes it natural to suggest global definition, using any distance $d$ over the set of finite words: 
\begin{defe}{\label{Généralisation}}
We define the centred pseudo-metric, denoted by $\besi{d}$, as follows: 
$$ \besi{d}(x,y)=\limsup_{l\rightarrow \infty} \dfrac{d(x_{\co{0}{l}},y_{\co{0}{l}})}{max_{u,v\in A^l}d(u,v) }, \forall x,y\in A^\N.$$
\end{defe}
On the other hand, a similar pseudo-metric known as Weyl pseudo-metric, also based on $\dhh$, measures the density of differences between two given sequences in arbitrary segments of given length. A general definition, based on any distance $d$ over finite words, would become:
\begin{defe}
We define the sliding pseudo-metric, denoted by $\weyl{d}$, as follows:
$$\weyl{d}(x,y)=\limsup_{l\rightarrow \infty} \max_{k\in \N}\dfrac{d(x_{\co{k}{k+l}},y_{\co{k}{k+l}})}{max_{u,v\in A^l}d(u,v) }, \forall x,y\in A^\N.$$
\end{defe}
The Weyl space shares many properties with the Besicovitch space; one of the main differences is that it is not complete, according to \cite{downarowicz1988quasi}; in terms of dynamics an open question is whether there exists an expansive cellular automaton over the Weyl space.
A relevant question is now the following:
Which properties of distance $d$ make dill map well-defined in the corresponding pseudometrics?

Generalizations exist of Besicovitch pseudometrics over groups (see for instance \cite{lacka_quasi-uniform_2016,capobianco_characterization_2020}). An interesting work would be to generalize more of these metrics to this setting.
Let us replace $(\N,+)$, generated by $1$, by a monoid $(\M,\cdot)$ generated by a finite set $\G$, that we assume without torsion element: for every $g\in\M$, $\sett{g^k}{k\in\N}$ is infinite.
A \emph{pattern} with finite support $U\subset\M$ is $u\in A^U$.
If $U=g\cdot V$, then the \emph{translate} by $g\in\M$ of a pattern $u$ is the pattern $\sigma^g(u)$ defined over support $V$ such that $\sigma^g(u)_i=u_{g\cdot i}$.
The \emph{deletion} $D^g_j$ at position $j\in\M$ with respect to generator $g\in\G$ is the function mapping any pattern $u$ with some support $U$ into the pattern $v$ defined over support $V=U\setminus\sett{jg^k}{jg^{k+1}\notin U}\cup\sett{jg^k}{jg^{k+1}\in U}$ by $v_{jg^k}=u_{jg^{k+1}}$, $v_{i}=u_i$ otherwise.
By torsion-freeness, $\card V=\card{U\setminus\{i\}}$.
Now one can consider, as a variant of the Levenshtein, the following:
the distance $\dmm(u,v)$ between patterns $u$ and $v$ is the minimal number $m+m'$ of deletions such that $D_{j_1}^{g_1}\circ D_{j_2}^{g_2}\circ\cdots D_{j_m}^{g_m}(u)$ and $D_{j'_1}^{g'_1}\circ D_{j'_2}^{g'_2}\circ\cdots D_{j'_{m'}}^{g'_{m'}}(u)$ have a common translate.
The Feldman-like pseudo-metric over configurations of $A^\M$, endowed with a spanning sequence $(\F_n)_{n\in\N}$ of finite sets of $\M$, would then be:
\[\dm=\limsup_{n\to\infty}\frac{\dmm(x_{\F_n},y_{\F_n}}{\card{\F_n}}
.\]
It is not clear exactly when this pseudo-metric would not behave pathologically, but probably F\o olner conditions should be assumed (see \cite{capobianco_characterization_2020}).
\bibliographystyle{alpha}
\bibliography{besiweyl}
\end{document}
[juste après a->aa;b->bb]
\pierre{bizarre : si $\tau$ est $\rho$-lipschitzien et $\tau^2$ est $1$-lipschitzien, alors pour tout $t$, \begin{align*}
		d(\tau^t(x),\tau^t(y))&\le d((\tau^2)^{\ipart{t/2}}(\tau^{t\bmod2}(x)),(\tau^2)^{\ipart{t/2}}(\tau^{t\bmod2}(y)))
		\\&\le 1^{\ipart{t/2}}\cdot d(\tau^{t\bmod2}(x),\tau^{t\bmod2}(y))
		\\&\le(max(1,\rho))\cdot d(x,y)
\end{align*}}